\documentclass{article}
\usepackage{hyperref}
\usepackage{amsmath, amssymb, amsthm}
\usepackage{authblk}
\usepackage{tikz}

\newtheorem{proposition}{Proposition}[section]
\newtheorem{claim}{Claim}[]
\newtheorem{case}{Case}[]
\newtheorem{subcase}{Sub-Case}[case]
\newtheorem{conjecture}{Conjecture}[section]
\newtheorem{corollary}{Corollary}[section]
\newtheorem{lemma}{Lemma}[section]

\newtheorem{theorem}{Theorem}[section]

\title{A note on the second neighborhood problem for $k$-anti-transitive and $m$-free digraphs}
\author{Dania Mezher\thanks{Corresponding Author: d.mezher@st.ul.edu.lb, dania.mezher.98@gmail.com}}
\author{Moussa Daamouch\thanks{moussa.daamouch@ul.edu.lb, moussadaamouch1923@gmail.com}} 
\affil{KALMA Laboratory, Department of Mathematics, Faculty of Sciences I, Lebanese University, Beirut, Lebanon}
 
\date{}

\begin{document}

\maketitle

\begin{abstract}
Seymour’s Second Neighborhood Conjecture (SSNC) asserts that every finite oriented
graph has a vertex whose second out-neighborhood is at least as large as its first outneighborhood.
Such a vertex is called a Seymour vertex. A digraph $D = (V, E)$ is $k$-anti-transitive if for every pair of vertices $u, v \in V$, the existence of a directed path of length $k$ from $u$ to $v$ implies that $(u, v) \notin E$.
An $m$-free digraph is digraph having no directed cycles with length at most $m$.
In this paper, we prove that if $D$ is $k$-anti-transitive and $(k-4)$-free digraph, then $D$ has a Seymour vertex. As a consequence, a special case of Caccetta-Haggkvist Conjecture holds on 7-anti-transitive oriented graphs. This work extends recently known results.
\end{abstract}
\section{Terminology and introduction}
In this paper, A directed graph (or a digraph) $D$ is a pair of two disjoint sets $(V, E)$, where $V$ is non-empty and $E \subset V \times V$. The set $V$, denoted by $V(D)$, is called the vertex set of $D$. While $E$, denoted by $E(D)$, is called the arc set of $D$. Each pair $(u, v)$ signifies an arc from a vertex $u$ to a vertex $v$. All the directed graphs discussed in this paper are finite oriented graph, meaning that $V$ is finite, there are no self-loops $(u, u) \notin E$, and for any pair of vertices $u$ and $v$ in $D$, there is at most one arc connecting them. We may write $u \rightarrow v$ meaning that $(u, v) \in E(D)$. We denote by $x_0 x_1 \cdots x_k$ a directed $x_0x_k$-path ($P_k$) and we may write $x_0 \rightarrow x_1 \rightarrow \cdots \rightarrow x_k$. A directed $k$-cycle ($C_k$) is denoted by $x_0 \cdots x_{k-1} x_0$, and we may write $x_0 \rightarrow \cdots \rightarrow x_{k-1} \rightarrow x_0$. The length of a directed path (or cycle) is the number of its arcs. The distance $\text{$dist$}(v, u)$ from a vertex $v$ to a vertex $u$ is the length of a shortest directed path from $v$ to $u$. The out-neighborhood $N_D^+(v)$ of a vertex $v$ is defined as the set of all vertices at distance 1 from $v$, that is
$N_D^+(v) = \{u \in V : \text{dist}(v,u)=1 \}$. The out-degree of a vertex $v$, denoted $d_D^+(v)$, is defined as $d^+_D(v) = |N^+_D(v)|$. The second out-neighborhood $N_D^{++}(v)$ of a vertex $v$ is defined as the set of all vertices $w$ at distance 2 from $v$, that is $N_G^{++}(v) = \{w \in V \colon \text{dist}(v, w) = 2\}$. The second out-degree of a vertex $v$, denoted $d_D^{++}(v)$, is defined as $d^{++}_D(v) = |N^{++}_D(v)|$. Let $\delta^+$ denote the minimum out-degree in $D$. Given a subset $R$ of vertices from $V(D)$, the subdigraph of $D$ induced by $R$ is represented as $D[R]$. The out-neighborhood of a vertex $v$ restricted to $R$ is: $N^+_R(v)  = N_D^+(v) \cap R$, and its corresponding out-degree within $R$ is: $d^+_R(v) = |N^+_R(v)|$. We may write $\overline{R}=V(D) \setminus R$. For non adjacent vertices $u$ and $v$, a directed $uv$-path of length $k$ will be denoted by $P^*_k$.\\

In 1990, Paul Seymour proposed the following conjecture.

\begin{conjecture}[Seymour's Second Neighborhood Conjecture]
For any oriented simple graph $D$, there is a vertex $v \in V$ such that $d^+(v) \leq d^{++}(v)$.
\end{conjecture}
\noindent
A vertex $v$ satisfying $d^+(v) \leq d^{++}(v)$ is referred to as a \emph{Seymour vertex}.\\

In the study of graph theory, when researchers encounter a challenging problem, it is a common strategy to limit the problem to a specific family of graphs or digraphs. Seymour's conjecture, as an instance, has been proven exclusively for a few selected classes of digraphs. Notably, in 1996, Fisher~\cite{fisher1996squaring} took a significant step by proving SSNC for a family of digraphs known as tournaments. Progress continued into the new millennium: in 2000, Havet and Thomassé~\cite{havet2000median} offered a streamlined proof of SSNC for tournaments, and in 2001, Kaneko and Locke~\cite{kaneko2001minimum} established the following theorem.

\begin{theorem}
\label{t.1}
\textbf{~\cite{kaneko2001minimum}:} Let $D$ be an oriented graph, and let $\delta^+$ be its minimum out-degree. If $\delta^+ \leq 6$, then $D$ has a Seymour vertex.
\end{theorem}
In 2007, Fidler and Yuster~\cite{fidler2007remarks} substantiated the case of tournaments missing a matching. Moving forward, Ghazal~\cite{ghazal2012, ghazal2013contribution, ghazal2015remark}, during the years 2012, 2013, and 2015, affirmed that the conjecture stands true for several classes of directed graphs, including tournaments missing an \( n \)-generalized stars and other categories of oriented graphs. In the same manner, some attempts were made to settle SSNC for tournaments missing a specific graph (see~\cite{al2021second,daamouch2024second,dara2022extending}).\\

Another approach to SSNC is to determine the maximum value $\lambda$ such that every oriented graphs has a vertex $v$ and $d^{++}(v) \geq \lambda d^+(v)$. In 2003, Chen, Shen, and Yuster \cite{chen2003second} proved that in every oriented graph, there exists a vertex $v$ such that $d^{++}(v) \geq \lambda d^+(v)$ where $ \lambda = 0.657\cdots$ is the unique real root in the interval $(0,1)$ of $2x^3+x^2-1=0$. In 2017, Liang and Xu \cite{liang2017seymour} improved this result for the class of digraphs having no directed cycles with length at most $m$, called $m$-free digraphs, by showing that for any $m$-free digraph $D$, there exists $v \in V(D)$ and a real number $\lambda_m \in (0,1)$ such that $d^{++}(v) \geq \lambda_m d^+(v)$, and $\lambda_m \rightarrow 1$ while $m \rightarrow \infty$. Note that the class of $m$-free digraphs is a subclass of digraphs with forbidden $C_m$, and SSNC for both still remains open.\\

In a separate development in 2012, Galena-Sánchez and Hernández-Cruz~\cite{hernandez2012k} introduced the class of \( k \)-transitive digraphs and the class of \( k \)-quasi-transitive digraphs. A digraph \( D \) is \( k \)-transitive  if, for any vertices \( u, v \) in \( D \), the existence of a directed path from \( u \) to \( v \) of length \( k \) implies that there is a arc from \( u \) to \( v \) in \( D \). A digraph \( D \) is \( k \)-quasi-transitive if, for any \( u, v \) in \( D \), the existence of a directed path from \( u \) to \( v \) of length \( k \) implies that there is a an arc from \( u \) to \( v \) or from \( v \) to \( u \) in \( D \). In terms of forbidden (not necessarily induced) subdigraphs, $k$-transitive oriented graphs are oriented graphs with forbidden $P_k^*$ and $C_{k+1}$, and $k$-quasi-transitive oriented graphs are oriented graphs with forbidden $P_k^*$. In 2017, García-Vásquez and Hernández-Cruz~\cite{garcia2017some} demonstrated that every 4-transitive oriented graph has a Seymour vertex. Additionally, Gutin and Li~\cite{gutin2017seymour} proved SSNC for quasi-transitive oriented graphs, specifically those that are 2-quasi-transitive. Moving ahead, recent works, as seen in~\cite{daamouch2020seymour, Daamouch2024}, validated SSNC for 3-quasi-transitive oriented graphs as well as for some instances of \( k \)-transitive oriented graphs, specifically for $k \leq \delta^+ +4$ and consequently for $k \leq 11$ as well as for $k$-transitive digraphs that are $(\frac{k}{2}-1)$-free.\\

The vertices $x$, $y$, and $z$ are said to form a transitive triangle if we have $x \rightarrow y \rightarrow z$ and also $x \rightarrow z$. It is easily seen that a transitive triangle is obtained from a directed triangle $C_3$ by reversing the direction of exactly one arc. We denote by $C(n,1)$ an oriented cycle on $n+1$ vertices obtained from a directed cycle $C_{n+1}$ by reversing the direction of exactly one arc. A digraph $D$ is \emph{$k$-anti-transitive} if for any $u, v \in V(D)$, the existence of a directed $uv$-path of length $k$ implies $(u,v) \notin E(D)$. In other words, the class of $k$-anti-transitive digraphs are that of digraphs with forbidden $C(k,1)$. In 2009, Brantner, Brockman, Kay, and Snively~\cite{brantner2009contributions} proved that every minimum out-degree vertex in 2-anti-transitive digraphs (digraphs without transitive triangles) is a Seymour vertex. In 2020, this result is improved and extended in~\cite{daamouch2020ant,daamouch2021seymour} to $k$-anti-transitive digraphs for $k=\{3,4,5\}$. later, in 2021, Hassan et al.~\cite{hassan2021seymour} proved SSNC for 6-anti-transitive digraphs. It is seen that the difficulty of SSNC for $k$-anti-transitive digraphs is increasing with respect to $k$, especially with the lack of structural description of these digraphs. Note that every digraph $D$ is $k$-anti-transitive for $k$ greater than the length of a longest directed path in $D$. So if a property holds for $k$-anti-transitive digraphs with arbitrary $k$, then this property holds for all digraphs. As it is difficult to settle SSNC for the class of $k$-anti-transitive digraphs when $k$ is an arbitrary positive integer, it is convenient to combine this class of digraphs with another. For instance, in~\cite{daamouch2021seymour}, the SSNC is proven for digraphs that are both $k$-anti-transitive and $(k-3)$-free. In this paper, we prove SSNC for digraphs that are $k$-anti-transitive and $(k-4)$-free. As a consequence, we show that a special case of Caccetta-Haggkvist Conjecture holds on 7-anti-transitive oriented graphs.
\section{Some requisite results}
\begin{lemma}
\label{l.1}
~\cite{daamouch2021seymour}
Let $D$ be an oriented graph, and let $\delta^+$ be its minimum out-degree. Let $v \in V(D)$ such that $d^+(v) = \delta^+$. If $D[N^+(v)]$ has a sink vertex $x$, i.e., $d^{+}_{N^+(v)}(x) = 0$, then $v$ is a Seymour vertex.
\end{lemma}
\begin{proof}
We have $|N^{++}(v)| \geq |N^+(x) \setminus N^+(v)| = |N^+(x)| \geq \delta^+ = |N^+(v)|$.
\end{proof}

\begin{lemma}
\label{l.2}
~\cite{daamouch2020ant}
Let $D$ be a $k$-anti-transitive oriented graph with $k \geq 2$, and let $v \in V(D)$. If there exists a directed path of length $d$ in $D[N^+(v)]$, then $d \leq k - 2$.
\end{lemma}
\begin{proof}
Suppose to the contrary that $d > k - 2$. Then we can find a directed path $x_0  x_1 \cdots  x_{k-1} $ of length $k - 1$ in $D[N^+(v)]$. Hence, $ v x_0 \cdots x_{k-1} $ is a directed path of length $k$, but $D$ is $k$-anti-transitive, so $v \nrightarrow x_{k-1}$, which contradicts to $x_{k-1} \in N^+(v)$.
\end{proof}

\begin{lemma}
\label{l.3}
~\cite{hassan2021seymour}
Let $v$ be a vertex of $D$ with the minimum out-degree, i.e., $d^+(v) = \delta^+$, and $R = N^+(v)$. Let $d^{++}(v) \leq \delta^+ - 1$. For distinct $r_i, r_j, r_k \in R$, we have
\[
\begin{aligned}
|N^+_{\overline{R}}(r_i, r_j)| &\geq \delta^+ + 1 - d^+_R(r_i) - d^+_R(r_j) \\
|N^+_{\overline{R}}(r_i, r_j, r_k)| &\geq \delta^+ + 2 - d^+_R(r_i) - d^+_R(r_j) - d^+_R(r_k)
\end{aligned}
\]
\end{lemma}

\begin{proposition}
\label{p.1}
~\cite{daamouch2021seymour}
Let \( D \) be an \( m \)-free digraph. If \( m \geq \delta^+ - 1 \), then \( D \) has a Seymour vertex.
\end{proposition}
\noindent
\textbf{Remark:} Let $D$ be an oriented simple graph. In view of Theorem~\ref{t.1}, when proving that $D$ has a Seymour vertex, it is sufficient to assume that $\delta^+ \geq 7$.\\

It is well known that SSNC implies the following special case of Caccetta-Haggkvist Conjecture. 

\begin{conjecture}
~\cite{caccetta1978}
Every digraph on $n$ vertices with minimum in-degree and minimum out-degree both at least $\frac{n}{3}$ has a directed triangle $C_3$.
\end{conjecture}

More precisely, we have the following.

\begin{proposition}
\label{p.2}
Let $D$ be an oriented graph on $n$ vertices such that its minimum in-degree and its minimum out-degree are both at least $\frac{n}{3}$. If $D$ has a Seymour vertex, then it has a directed triangle.
\end{proposition}

As SSNC holds on $k$-anti-transitive digraphs for $2 \leq k \leq 6$, the special case of Caccetta-Haggkvist conjecture also holds on $k$-anti-transitive digraphs for $2 \leq k \leq 6$.
\section{Main result}

\begin{theorem}
\label{t.main} 
Let $D$ be an oriented graph. If \( D \) is both \( k \)-anti-transitive and $(k-4)$-free, then \( D \) has a Seymour vertex.
\end{theorem}
\begin{proof}
Notice that the validity of this theorem has been confirmed for $k=6$, denoting $6$-anti-transitivity and $2$-freeness~\cite{hassan2021seymour}. Subsequently, our focus shifts to cases with $k \geq 7$, ensuring that the digraph is at least $3$-free. Hence $m \geq 3$. In addition, in view of Theorem~\ref{t.1}, it is sufficient to assume that $\delta^+ \geq 7$.\\
Let \( v \in V(D) \) such that \( d^+(v) = \delta^+ \). Set \( R = N^+(v) = \{r_0, r_1, \ldots, r_{\delta^+-1}\} \) and let $m=k-4$.
If $m \geq \delta^+ - 1$, then by Proposition~\ref{p.1}, \( D \) has a Seymour vertex. Due to this, we can presume that \( m < \delta^+ - 1 \).
If there exists$r \in R$ such that $d^+_R(r)=0$ (this vertex is referred to as a sink vertex), then by Lemma~\ref{l.1}, $v$ is a Seymour vertex in $D$. Presume now that we have $ d^+_R(r) \geq 1 $ for all \( r \in R \). Let \( l \) be the length of a longest directed path in \( D[R] \). Given that \( D \) is an \( m \)-free digraph, we have \( D[R] \) is also an \( m \)-free digraph, and as \( d^+_R(r) \geq 1 \) for all \( r \in R \), it follows that \( D[R] \) has a directed path of length at least \( m \). On the other hand, Lemma~\ref{l.2} implies that \( l \leq k - 2 = (m+4) - 2 = m +2\). So we have three cases: \( l = m \), \( l = m + 1 \) and \(l=m+2\). \\Observe that we may assume that $d^{++}(v) \leq \delta^+ - 1$, since otherwise $v$ is a Seymour vertex.
\begin{case}
$\ell = m$.
\end{case}
Let $r_0r_1\cdots r_m$ be a longest path in $D[R]$ of length $m$. Set \(B =\{r_i\in R \colon i \in \{0,\ldots,m\}\}\). As $r_0r_1\cdots r_m$ is a longest path, we must have $N^+_R(r_m) \subseteq B$. If $r_m \rightarrow r_i$ for some $i\neq 0$, then $r_m r_i \cdots r_m$ would be a directed cycle of length at most $m$, which contradicts the fact that \(D\) is an \(m\)-free digraph. It follows that $N^+_R(r_m) = \{r_0\}$ as $d^+_R(r_m)\geq 1$, and hence $d^+_R(r_m) = 1$. Similarly, since $D$ is an $m$-free digraph and $\ell=m$, we must have $N^+_R(r_i) = \{r_{i+1}\}$ and for all $ i \in \{0,\ldots,m-1\}$. Therefore $d^+_R(r_i) = 1$ for all $ i \in \{0,\ldots,m\}$.
If $d^+(r_i) \geq \delta^+ + 1$ for some $ i \in \{0,\ldots,m\}$, then $d^{++}(v) \geq |N^+_{\overline{R}}(r_i)| \geq d^+(r_i) - d^+_R(r_i) \geq \delta^+ + 1 - 1 = \delta^+$, which contradicts our assumption that $d^{++}(v) \leq \delta^+ - 1$. Thus $d^+(r_i)=\delta^+$ for all $i \in \{0,\ldots,m\}$. It follows that $d^+_{\overline{R}}(r_i) = \delta^+ -1$ for all $i \in \{0,\ldots,m\}$. We can easily see that $N^+_{\overline{R}}(r_i) = N^+_{\overline{R}}(r_j)$ for all $i, j \in \{0,\ldots,m\}$. Otherwise, there exists a vertex $z \in \overline{R}$ such that $ z \in N^+_{\overline{R}}(r_i) \setminus N^+_{\overline{R}}(r_j)$. Thus, $|N^{++}(v)| \geq |N^+_{\overline{R}}(r_j)| + |\{z\}| \geq \delta^+ -1 +1 = \delta^+$, which contradicts the fact that $d^{++}(v) \leq \delta^+ - 1$. Now let $X = N^+_{\overline{R}}(r_m) = N^+_{\overline{R}}(r_0) $. We will show that there exists a vertex $z \in X$ such that $d^+_X(z) =0$, and $r_m$ will be a Seymour vertex. On the contrary, assume that for each vertex $z \in X$, we have $d^+_X(z) \geq 1$. Hence we can find $z_1,z_2,$ and $z_3 \in X$ such that $z_1 \rightarrow z_2 \rightarrow z_3$. As $m \geq 3$, we have $z_3 \nrightarrow z_1$. Furthermore $z_3 \nrightarrow z_4$ for any $z_4 \in X$. Because otherwise $r_0 r_1 \cdots r_m z_1 z_2 z_3 z_4$ would be a directed path of length \(k=m+4\), giving \(z_4 \notin N^+(r_0) \), which is a contradiction as $z_4 \in N^+(r_0)$. Thus, there exists $z \in X$ such that $d^+_X(z) = 0$, and hence $d^+_{\overline{X}}(z) \geq \delta^+$. Therefore $d^{++}(r_m) \geq d^+_{\overline{X}}(z) \geq \delta^+ = d^+(r_m)$, and $r_m$ is a Seymour vertex.

\begin{case}
$\ell = m+1$.
\end{case}
Let $ r_0 r_1 \cdots r_{m+1}$ be a longest directed path in $D[R]$ of length $m+1$, and consider \( B = \{ r_i \in R \colon i \in \{0,\ldots,m+1\} \} \). It is clear that $N^+_R(r_{m+1}) \subseteq \{r_0, r_1\}$. In fact, suppose $r_{m+1} \rightarrow r_i$ for some $r_i \in B \setminus \{r_0,r_1\}$. Hence $r_{m+1} r_i \cdots r_{m+1}$ would be a directed cycle of length at most \(m\), which contradicts the fact that \(D\) is an \(m\)-free digraph. Clearly, $r_{m+1} \nrightarrow r$ for all $r \in R \setminus B $, otherwise $r_0 r_1 \cdots r_{m+1} r $ would be a directed path of length $m+2$, which is a contradiction. We examine two cases: $r_{m+1} \rightarrow r_0$ and $r_{m+1} \nrightarrow r_0$.

\begin{subcase}
$r_{m+1} \rightarrow r_0$.
\end{subcase}
In this case $r_0 r_1 \cdots r_{m+1}r_0$ is a directed cycle of length $m+2$. Note that for all $i \in \{0,\ldots,m\}$, we have $r_i \nrightarrow r$ for all $r \in R \setminus B$, since otherwise $r_{i+1} \cdots r_{m+1}r_0 \cdots r_i r $ would be a directed path of length $m+2$, which is a contradiction. In the other hand, since $D$ is an $m$-free digraph, we must have $N^+_R(r_i)\subseteq \{r_{i+1},r_{i+2}\}$ for all $r_i \in B$ where the subscripts are taken modulo $m+2$. Hence, we have $d^+_R(r_i) \leq 2$ for all $r_i \in B$. Observe that there must exist a vertex $r_i \in B$ such that $d^+_B(r_i) = 1$. Otherwise, $N^+_B(r_i)=\{r_{i+1},r_{i+2}\}$ for all $r_i \in B$, and hence $ r_{m+1} r_0 r_2 r_4 r_6 r_7 \cdots r_m r_{m+1} $ would be a directed cycle of length at most \(m\), which contradicts the fact that \(D\) is an \(m\)-free digraph. We can presume that $d^+_B(r_{m+1}) = 1$. It follows that $N^+_R(r_{m+1}) = \{r_0\}$ and $d^+_R(r_{m+1}) = 1$. Furthermore, we must have $d^+(r_{m+1}) = \delta^+$, for otherwise we get $d^{++}(v) \geq \delta^+$, which is a contradiction. Now, Lemma~\ref{l.3} implies that $|N^+_{\overline{R}}(r_1,r_m, r_{m+1})| \geq \delta^+ + 2 - d^+_R(r_1) - d^+_R(r_m) - d^+_R(r_{m+1}) \geq \delta^+ +2-2-2-1 \geq \delta^+ -3 \geq 4$. So we can find a vertex $z_1$ such that $z_1 \in N^+_{\overline{R}}(r_1, r_m, r_{m+1})$. Let $X = N^+(r_{m+1})$.\\
If $d^+_X(z_1) \leq 1$, then $d^+_{\overline{X}}(z_1) \geq \delta^+ -1$. Observe that $r_1 \in N^{++}(r_{m+1})$ and $r_1 \notin N^+_{\overline{X}}(z_1)$. It follows that $d^{++}(r_{m+1}) \geq d^+_{\overline{X}}(z_1) + |\{r_1\}| \geq \delta^+ -1 +1 = \delta^+ = d^+(r_{m+1})$, and hence $r_{m+1}$ is a Seymour vertex.\\
If $d^+_X(z_1) \geq 2$, then there must exist a vertex $z_2 \in N^+_{\overline{R}}(r_{m+1})$ such that $z_1 \rightarrow z_2$. Note that $z_2 \nrightarrow z_3$ for any $z_3 \in N^+_{\overline{R}}(r_{m+1})$, since otherwise $r_{m+1} r_0 r_1 \cdots r_m z_1 z_2 z_3 $ would be a directed path of length \(k=m+4\), such that $r_{m+1} \rightarrow z_3$ which is a contradiction. Additionally, $z_2 \nrightarrow r_0$ since otherwise $ v r_{m+1} z_1 z_2 r_0 \cdots r_m$ would be a directed path of length \(k=m+4\), giving \(r_m \notin N^+(v) \) , which is a contradiction.
Hence, $d^+_X(z_2) = 0 $ and $d^+_{\overline{X}}(z_2) \geq \delta^+$. 
Therefore $d^{++}(r_{m+1}) \geq d^+_{\overline{X}}(z_2) \geq \delta^+ = d^+(r_{m+1})$, and $r_{m+1}$ is a Seymour vertex.

\begin{subcase}
$r_{m+1} \nrightarrow r_0$.
\end{subcase}
In this case we have $N^+_R(r_{m+1})=\{r_1\}$ and $d^+_R(r_{m+1}) = 1$, since otherwise $d^+_R(r_{m+1})=0$, which is a contradiction. In addition, we may assume that $d^+(r_{m+1}) = \delta^+$, since otherwise $d^+(r_{m+1}) \geq \delta^+ + 1$ and $d^{++}(v) \geq \delta^+$, which is a contradiction. 

\begin{claim}
\label{cl.1}
$N^+_{R \setminus B }(r_m) = \emptyset$, unless $r_{m+1}$ is a Seymour vertex.
\begin{proof}[Proof of Claim~\ref{cl.1}]
Assume that there exists a vertex $r_{m+2} \in N^+_{R\setminus B}(r_m) $. Clearly, $r_{m+2} \nrightarrow r_0$, as otherwise $r_{m+2} r_0 r_1 \cdots r_{m+1}$ would be a directed path of length $m+2$ which is a contradiction. Using the same previous verification, we get that  $N^+_R(r_{m+2}) = \{r_1\}$. Let $z_1 \in N^+_{\overline{R}}(r_m,r_{m+1}, r_{m+2})$, such a vertex exists since Lemma~\ref{l.3} implies that $|N^+_{\overline{R}}(r_m,r_{m+1}, r_{m+2})| \geq \delta^+ + 2 - (\delta^+ -2)- 1 - 1 = 2$. We have $z_1 \nrightarrow r_2$ since otherwise $ r_2 \cdots r_m z_1 r_2$ would be a directed cycle of length at most \(m\), which contradicts the fact that \(D\) is an \(m\)-free digraph. Set $X = N^+(r_{m+1})$.\\
If $d^+_X(z_1) \leq 1$, then $d^+_{\overline{X}}(z_1) \geq \delta^+ -1$. Thus, $d^{++}(r_{m+1}) \geq d^+_{\overline{X}}(z_1) + |\{r_2\}| \geq \delta^+ -1 +1 = \delta^+ = d^+(r_{m+1})$ and $r_{m+1}$ is a Seymour vertex.\\
If $d^+_X(z_1) \geq 2 $, then $z_1 \rightarrow z_2$ for some $z_2 \in N^+_{\overline{R}}(r_{m+1})$. Note that $z_2 \nrightarrow z_3$ for any $z_3 \in N^+_{\overline{R}}(r_{m+1})$, for otherwise $r_{m+1} r_1 \cdots r_m r_{m+2} z_1 z_2 z_3$ is a directed path of length \(k=m+4\) such that \(r_{m+1} \rightarrow z_3\), which is a contradiction. Additionally, $z_2 \nrightarrow r_1$ since otherwise $v r_{m+2} z_1 z_2 r_1 \cdots r_{m+1}$ is a directed path of length $m+4=k$ such that $v \rightarrow r_{m+1}$, which is a contradiction. Thus, $d^+_X(z_2)=0$ and $d^+_{\overline{X}}(z_2) \geq \delta^+$. Therefore, $d^{++}(r_{m+1}) \geq d^+_{\overline{X}}(z_2) \geq \delta^+ = d^+(r_{m+1})$, and $r_{m+1}$ is a Seymour vertex.\\
Finally, we proved that $r_{m+1}$ is Seymour vertex unless $r_m \nrightarrow r$ for all $r\in~{R \setminus B}$.
\end{proof}
\end{claim}

By Claim~\ref{cl.1}, we may assume that $N^+_{R \setminus B }(r_m) = \emptyset$; this means that $N^+_R(r_m)\subseteq \{r_0,r_{m+1}\}$ and $d^+_R(r_m) \leq 2$. Now, observe that $d^+_R(r_1) \leq \delta^+ - 3$ as $r_0,r_{m+1} \notin N^+_R(r_1)$.
Hence, by applying Lemma~\ref{l.3}, we have $|N^+_{\overline{R}}(r_1,r_m, r_{m+1})|\geq \delta^+ + 2 - (\delta^+ - 3) - 2 - 1 = 2$. Let $z_1 \in N^+_{\overline{R}}(r_1,r_m, r_{m+1})$. Set $X = N^+(r_{m+1})$. It is clear that $r_2 \in N^{++}(r_{m+1})$. Note that $z_1 \nrightarrow r_2$ since otherwise $ r_2 \cdots r_m z_1 r_2$ would be a directed cycle of length at most \(m\), which contradicts the fact that \(D\) is an \(m\)-free digraph.\\
If $d^+_X(z_1) \leq 1$, then $d^+_{\overline{X}}(z_1) \geq \delta^+ -1$. Therefore, $d^{++}(r_{m+1}) \geq d^+_{\overline{X}}(z_1) + |\{r_2\}| \geq \delta^+ = d^+(r_{m+1})$, and $r_{m+1}$ is a Seymour vertex.\\
If $d^+_X(z_1) \geq 2 $, then there exists $z_2 \in  N^+_{\overline{R}}(r_{m+1})$ such that $z_1 \rightarrow z_2$. Clearly, if $d^+_X(z_2)=0$ then $r_{m+1}$ is a Seymour vertex by Lemma~\ref{l.1}. So, we may assume that $d^+_X(z_2)\geq 1$. Note that $z_2 \nrightarrow r_1$, for otherwise $r_1z_1z_2r_1$ is a $C_3$, which is a contradiction. Hence there exists $z_3 \in  N^+_{\overline{R}}(r_{m+1})$ such that $z_2 \rightarrow z_3$. Note that $z_3 \nrightarrow z_4$ for any $z_4 \in N^+_{\overline{R}}(r_{m+1})$, for otherwise $r_{m+1} r_1 \cdots r_m z_1 z_2 z_3 z_4$ would be a directed path of length \(k=m+4\), giving \(z_4 \notin N^+(r_{m+1})\), which is a contradiction. Additionally, $z_3 \nrightarrow r_1$, otherwise $v r_2 \cdots r_{m+1} z_1 z_2 z_3 r_1$ would extend to a $C(k,1)$, which is a contradiction.
Therefore, $d^+_X(z_3) = 0$ and $d^+_{\overline{X}}(z_3) \geq \delta^+ $.
Hence $r_{m+1}$ is a Seymour vertex as $d^{++}(r_{m+1}) \geq d^+_{\overline{X}}(z_3)  \geq \delta^+ = d^+(r_{m+1})$.

\begin{case}
$\ell = m+2$.
\end{case}
Let $ r_0 r_1 \cdots r_{m+2}$ be a longest directed path in $D[R]$ of length $m+2$. Set \( B = \{ r_i \in R \colon i \in \{0,\ldots,m+2\} \} \).
In this case we must have $N^+_R(r_{m+2}) \subseteq \{r_0, r_1, r_2\}$. We examine two cases: $r_{m+2} \rightarrow r_0$ and $r_{m+2} \nrightarrow r_0$.

\begin{subcase}
\( r_{m+2} \to r_0 \).
\end{subcase}
Note that for all $i \in \{0,\ldots,m+1\}$, we have $r_i \nrightarrow r$ for all $r \in R \setminus B$, since otherwise $r_{i+1} \cdots r_{m+2} r_0 \cdots r_i r $ would be a directed path of length $m+3$, which is a contradiction. There must exist a vertex $r_i \in B$ such that \( d^+_B(r_i) \leq~2\), because otherwise \( d^+_B(r_i) = 3 \) and $ N^+_B(r_i)=\{r_{i+1},r_{i+2},r_{i+3}\}$ for all $ i \in \{0,\ldots,m+2\}$ $\mod(m+2)$, hence $ r_{m+2} r_0 r_3 r_6 \cdots r_{m+2} $ would be a directed cycle of length at most \(m\), which contradicts the fact that \(D\) is an \(m\)-free digraph. We may assume that \( d^+_R(r_{m+2}) \leq 2 \); this means that either $r_{m+2}\rightarrow r_1$ or $r_{m+2}\rightarrow r_2$. If \( d^+(r_{m+2}) \geq \delta^+ + 2 \), then $d^{++}(v)\geq \delta$, which is a contradiction. Thus, \( d^+(r_{m+2}) \leq \delta^+ + 1 \). Let $X = N^+(r_{m+2})$. There exists \( z_1 \in N^+_{\overline{R}}(r_1, r_{m+1}, r_{m+2}) \), since \( |N^+_{\overline{R}}(r_1, r_{m+1}, r_{m+2})| \geq \delta^+ + 2 - 3 - 3 - 2 \geq 1 \), by Lemma~\ref{l.3}.
We have \( z_1 \nrightarrow z_2 \) for any \( z_2 \in N^+_{\overline{R}}(r_{m+2}) \), Since otherwise  \(  r_{m+2} r_0 r_1 \cdots r_{m+1} z_1 z_2 \) is a directed path of length \(k=m+4\), giving \(z_2 \notin N^+(r_{m+2})\), which is a contradiction. Additionally \( z_1 \nrightarrow r_0\) and \(z_1 \nrightarrow r_2\), as otherwise \( v r_{m+2} z_1 r_0 r_1 \cdots r_{m+1}\) and \(v r_{m+2} r_0 r_1 z_1 r_2 \cdots r_{m+1} \)  are directed paths of length \(k=m+4\), giving \(r_{m+1} \notin N^+(v) \) which is a contradiction. Notice that $z_1 \nrightarrow r_1$. Therefore, $d^+_X(z_1)=0$ and $d^+_{\overline{X}}(z_1) \geq \delta^+$. Observe that $\{r_1,r_2\}\cap N^{++}(r_{m+2})\neq \emptyset$. If $r_1 \in N^{++}(r_{m+2})$, then $d^{++}(r_{m+2}) \geq d^+_{\overline{X}}(z_1) + |\{r_1\}| \geq \delta^+ +1 = d^+(r_{m+2})$, and $r_{m+2}$ is a Seymour vertex. Finally, if $r_2 \in N^{++}(r_{m+2})$, then $d^{++}(r_{m+2}) \geq d^+_{\overline{X}}(z_1) + |\{r_2\}| \geq \delta^+ +1 = d^+(r_{m+2})$, and $r_{m+2}$ is a Seymour vertex.

\begin{subcase}
\( r_{m+2} \nrightarrow r_0 \).
\end{subcase}
In this case we have $N^+_R(r_{m+2}) \subseteq \{r_1, r_2\}$, and hence $d^+_R(r_{m+2}) \leq 2$.\\
In order to find $z \in N^+_{\overline{R}}(r_2,r_{m+1},r_{m+2})$, we need upper bounds for $d^+_R(r_2)$ and $d^+_R(r_{m+1})$ to ensure that $|N^+_{\overline{R}}(r_2,r_{m+1},r_{m+2})|~\neq~0$. We establish this by proving that if \( r_2 \) or \( r_{m+1} \) have an out neighbor in \( R \setminus B \), then \(r_{m+2} \) is a Seymour vertex.

\begin{claim}
\label{cl.2}
$N^+_{R \setminus B }(r_{m+1}) = \emptyset$, unless $r_{m+2}$ is a Seymour vertex.
\begin{proof}[Proof of Claim~\ref{cl.2}]
Assume that there exists a vertex $r_{m+3}\in N_{R\setminus B}^+(r_{m+1})$. Clearly, we have $r_{m+3}\nrightarrow r_0$ as otherwise $ v r_{m+3} r_0 \cdots r_{m+2} $ is a directed path of length $k=m+4$, giving $r_{m+2} \notin R$, which is a contradiction.
As $\ell=m+2$ and $D$ is an $m$-free digraph, we get $N^+_R(r_{m+3}) \subseteq \{r_1, r_2\}$. Similarly, for all $r\in N_{R\setminus B}^+(r_{m+1})$, we have $N^+_R(r) \subseteq \{r_1, r_2\}$.

Suppose first that there exists a vertex $r \in N^+_R(r_{m+1}) \setminus \{r_0\}$ such that $r \rightarrow r_1$. Without loss of generality, assume $r_{m+2} \rightarrow r_1$.
We have \( d^+(r_{m+2}) \leq \delta^+ + 1 \), for otherwise we get $d^{++}(v)\geq \delta^+$, which is a contradiction.
Let $X = N^+(r_{m+2})$.
Let $z_1 \in N^+_{\overline{R}}(r_{m+2}, r_{m+3})$, such a vertex exists since Lemma~\ref{l.3} implies that $|N^+_{\overline{R}}(r_{m+2}, r_{m+3})| \geq \delta^+ + 1 - 2 - 2 \geq 4$.
Note that, $r_2 \in N^{++}(r_{m+2})$ if $r_{m+2} \nrightarrow r_2$, and $r_3 \in N^{++}(r_{m+2})$ otherwise.
Observe that $z_1 \nrightarrow r_1$, $z_1 \nrightarrow r_2$, $z_1 \nrightarrow r_3$, and $z_1 \nrightarrow z_2$ for any $z_2 \in N^+_{\overline{R}}(r_{m+2})$,
since otherwise $ v r_{m+3} z_1 r_1 \cdots r_{m+2} $,  $ v r_{m+3} z_1 r_2 \cdots r_{m+2} r_1 $, 
$ v r_{m+3} z_1 r_3 \cdots r_{m+2} r_1 r_2 $, and $ r_{m+2} r_1 r_2\cdots r_{m+1} r_{m+3} z_1 z_2 $ would extend to a $C(k,1)$ respectively, which is a contradiction.
Thus, $d^+_X(z_1) = 0$.
Therefore, if $r_{m+2} \nrightarrow r_2$ then $d^{++}(r_{m+2}) \geq d^+_{\overline{X}}(z_1) + |\{r_2\}| \geq \delta^+ +1 \geq d^+(r_{m+2})$, and $r_{m+2}$ is a Seymour vertex.
Otherwise, $d^{++}(r_{m+2}) \geq d^+_{\overline{X}}(z_1) + |\{r_3\}| \geq \delta^+ +1 \geq d^+(r_{m+2})$, and $r_{m+2}$ is a Seymour vertex.

Suppose now that for all $r \in N^+_R(r_{m+1})\setminus \{r_0\}$, we have $r \nrightarrow r_1$.
If there exists $r \in N^+_R(r_{m+1}) \setminus \{r_0\}$ such that $r \nrightarrow r_2$, then $d^+_R(r)=0$, which is a contradiction. Thus for all $r \in N^+_R(r_{m+1}) \setminus \{r_0\}$, we must have $r \rightarrow~r_2$, and hence $N^+_R(r)=\{r_2\}$. If \( d^+(r_{m+2}) \geq \delta^+ + 1 \), then $d^{++}(v)\geq \delta^+$, which is a contradiction. Thus \( d^+(r_{m+2}) = \delta^+ \). Let $X = N^+(r_{m+2})$. Since $d^+_R(r_{m+1}) \leq \delta^+ - 2$, Lemma~\ref{l.3} implies that $|N^+_{\overline{R}}(r_{m+1}, r_{m+2}, r_{m+3})| \geq \delta^+ + 2 - (\delta^+ - 2) - 1 - 1 = 2$. Let $z_1 \in N^+_{\overline{R}}(r_{m+1}, r_{m+2}, r_{m+3})$.
Notice that $ z_1 \nrightarrow r_3$, for otherwise $ r_3 \cdots r_{m+1} z_1 r_3$ would be a directed cycle of length at most $m$, which is a contradiction.\\
If $d^+_X(z_1) \leq 1 $, then $d^+_{\overline{X}}(z_1) \geq \delta^+ -1$. Hence $d^{++}(r_{m+2}) \geq d^+_{\overline{X}}(z_1) + |\{r_3\}| \geq \delta^+ = d^+(r_{m+2})$, and $r_{m+2}$ is a Seymour vertex.\\
If $d^+_X(z_1) \geq 2$, then $z_1 \rightarrow z_2$ for some $z_2 \in N^+_{\overline{R}}(r_{m+2})$.
Note that $z_2 \nrightarrow z_3$ for any $z_3 \in N^+_{\overline{R}}(r_{m+2})$
since otherwise $r_{m+2} r_2\cdots r_{m+1} r_{m+3} z_1 z_2 z_3$ is a directed path of length \(k=m+4\), giving \(z_3 \notin N^+(r_{m+2})\), which is a contradiction. Therefore, $d^+_X(z_2) \leq 1$ and $d^+_{\overline{X}}(z_1) \geq \delta^+ - 1$.
Additionally, $z_2 \nrightarrow r_3$ since otherwise  $ v r_{m+3} z_1 z_2 r_3 \cdots r_{m+2} r_2 $ would extend to a $C(k,1)$, which is a contradiction.
Therefore, $d^{++}(r_{m+2}) \geq d^+_{\overline{X}}(z_2) + |\{r_3\}| \geq \delta^+ = d^+(r_{m+2})$, and $r_{m+2}$ is a Seymour vertex.\\
Finally, we proved that $r_{m+2}$ is a Seymour vertex, unless there is no vertex $ r \in N^+_{R \setminus B }(r_{m+1})$. 
\end{proof}
\end{claim}
By Claim~\ref{cl.2}, we may assume that $N^+_{R \setminus B }(r_{m+1}) = \emptyset$; which means that $N^+_R(r_{m+1})\subseteq \{r_0,r_1,r_{m+2}\}$, and hence $d^+_R(r_{m+1}) \leq 3$.

\begin{claim}
\label{cl.3}
$N^+_{R \setminus B }(r_2) = \emptyset$, unless $r_{m+2}$ is a Seymour vertex.
\begin{proof}
Suppose that there is a $r_{m+3} \in N^+_{R\setminus B}(r_2) $. Recall that $r_{m+2} \nrightarrow r_0$ and $d^+_R(r_{m+2}) \leq 2$. Hence, we may assume that $d^+(r_{m+2}) \leq \delta^+ + 1$.
We have $r_{m+3} \nrightarrow r_2$ and, as $\ell=m+2$, we have $r_{m+3} \nrightarrow r_0$. In addition, we have $r_{m+3} \nrightarrow r_3$, for otherwise $v r_0 r_1 r_2 r_{m+3} r_3 \cdots r_{m+2}$ would be a directed path of length $k=m+4$, giving that $r_{m+2} \notin N^+(v)$, which is a contradiction. Hence $r_0,r_2,r_3 \notin N^+_R(r_{m+3})$, and therefore $d^+_R(r_{m+3}) \leq \delta^+ - 4$. It follows, by Lemma~\ref{l.3}, that $|N^+_{\overline{R}}(r_{m+1},r_{m+2}, r_{m+3})| \geq \delta^+ + 2 -2 - 3 - (\delta^+ - 4) = 1$. So let $z_1\in N^+_{\overline{R}}(r_{m+1},r_{m+2}, r_{m+3})$ and let $X = N^+(r_{m+2})$.

Suppose first that $r_{m+2} \rightarrow r_2$. Hence $ r_3,r_{m+3} \in N^{++}(r_{m+2})$. We have $z_1 \nrightarrow r_{m+3}$, for otherwise $vr_1\cdots r_{m+2}z_1r_{m+3}$ would extend to a $C(k,1)$, which is a contradiction. Also, $z_1 \nrightarrow r_3$, for otherwise $r_3 \cdots r_{m+1} z_1 r_3$ would be a directed cycle of length at most $m$, a contradiction.\\
If $d^+_X(z_1) \leq 1$, then $d^+_{\overline{X}}(z_1) \geq \delta^+ -1$. It follows that $d^{++}(r_{m+2}) \geq d^+_{\overline{X}}(z_1) + |\{r_3,r_{m+3}\}| \geq \delta^+ -1 +2 = \delta^+ +1 \geq d^+(r_{m+2})$, and $r_{m+2}$ is a Seymour vertex.\\
If $d^+_X(z_1) \geq 2$, then there exists $z_2 \in X$ such that $z_1 \rightarrow z_2 $. We may assume that $z_2 \in N^+_{\overline{R}}(r_{m+2})$ because $z_1 \nrightarrow r_1$, for otherwise $v r_3 \cdots r_{m+2} z_1 r_1 r_2 r_{m+3}$ would extend to a $C(k,1)$, a contradiction. In addition, we have $z_2 \nrightarrow r_{m+3}$, since otherwise $v r_1\cdots r_{m+1} z_1 z_2 r_{m+3}$ would extend to a $C(k,1)$, a contradiction. Now if $d^+_X(z_2)=0$, then $d^{++}(r_{m+2}) \geq d^+_{\overline{X}}(z_2) + |\{r_{m+3}\}| \geq \delta^+ +1 \geq d^+(r_{m+2})$. Therefore $r_{m+2}$ is a Seymour vertex.
Otherwise, if $d^+_X(z_2) \geq 1$, then there exists $z_3 \in X$ such that $z_2 \rightarrow z_3$. Notice that $z_2 \nrightarrow r_1$ and $z_2 \nrightarrow r_2$, for otherwise $v r_4\cdots r_{m+2} z_1 z_2 r_1 r_2 r_{m+3}$ and $v r_3\cdots r_{m+2} z_1 z_2 r_2 r_{m+3}$ would extend to a $C(k,1)$, a contradiction. This means that we may assume $z_3 \in N^+_{\overline{R}}(r_{m+2})$. Observe that $z_3 \nrightarrow r_1$, $z_3 \nrightarrow r_2$ and $z_3 \nrightarrow z_4$ for any $z_4 \in N^+_{\overline{R}}(r_{m+2})$ since otherwise $ v r_3 \cdots r_{m+2} z_1 z_2 z_3 r_1 $, $ v r_3 \cdots r_{m+2} z_1 z_2 z_3 r_2 $, and $r_{m+2} r_2 \cdots r_{m+1} z_1 z_2 z_3 z_4$  would extend to a $C(k,1)$, a contradiction. Hence, $d^+_X(z_3) = 0$.
Moreover, $z_3 \nrightarrow r_3$, since otherwise $v r_{m+3} z_1 z_2 z_3 r_3 \cdots r_{m+2}$  would extend to a $C(k,1)$, which is a contradiction.
Therefore $d^{++}(r_{m+2}) \geq d^+_{\overline{X}}(z_3) +|\{r_3\}| \geq \delta^+ + 1 \geq d^+(r_{m+2})$, and $r_{m+2}$ is a Seymour vertex.

Suppose now that $r_{m+2} \nrightarrow r_2$. Accordingly, we must have $N^+_R(r_{m+2}) = \{r_1\}$ and $d^+(r_{m+2}) = \delta^+$.
Let $X = N^+(r_{m+2})$.
Recall that there exists $z_1 \in N^+_{\overline{R}}(r_{m+1},r_{m+2}, r_{m+3})$.\\
If $d^+_X(z_1) = 0$, then $d^{++}(r_{m+2}) \geq d^+_{\overline{X}}(z_1) \geq \delta^+ = d^+(r_{m+2})$, and $r_{m+2}$ is a Seymour vertex.\\
If $d^+_X(z_1) \geq 1$, then there exists $z_2 \in X$ such that $z_1 \rightarrow z_2$. We have $z_1 \nrightarrow~r_1$, for otherwise $v r_3 \cdots r_{m+2} z_1 r_1 r_2 r_{m+3}$ would extend to a $C(k,1)$, a contradiction. Thus $z_2 \in N^+_{\overline{R}}(r_{m+2})$. Now observe that, $z_2 \nrightarrow r_1$ and $z_2 \nrightarrow z_3$ for any $z_3 \in N^+_{\overline{R}}(r_{m+2})$, for otherwise $v r_3 \cdots r_{m+2} z_1 z_2 r_1 r_2$ and
$ r_{m+2} r_1 r_2 \cdots r_{m+1} z_1 z_2 z_3 $ would extend to a $C(k,1)$, which is a contradiction. Hence $d^+_X(z_2) = 0$, and therefore $r_{m+2}$ is a Seymour vertex.
Finally, we proved that $r_{m+2}$ is a Seymour vertex, unless there is no vertex $ r \in N^+_{R \setminus B }(r_2)$. 
\end{proof}
\end{claim}
By Claim~\ref{cl.3}, we may assume that $N^+_R(r_2) \subseteq \{r_3,r_4,r_5\}$, and hence $d^+_R(r_2) \leq~3$.
Now, we can find $z_1 \in N^+_{\overline{R}}(r_2,r_{m+1}, r_{m+2})$ because by Lemma~\ref{l.3}, Claims~\ref{cl.2} and~\ref{cl.3}, we have $N^+_{\overline{R}}(r_2,r_{m+1}, r_{m+2}) \geq \delta^+ + 2 - 3 - 3 - 2 \geq 1$.\\
Recall that. in this case, we have $N^+_R(r_{m+2}) \subseteq \{r_1, r_2\}$ and $d^+_R(r_{m+2}) \leq 2$. To finish the proof, we consider the remaining two cases $d^+_R(r_{m+2}) = 1$ and $d^+_R(r_{m+2}) = 2$.\\

1) Suppose first that $d^+_R(r_{m+2}) = 1$.\\
Hence $N^+_R(r_{m+2}) = \{r_i\}$ for $i \in\{1,2\}$. In this case, we have $d^+(r_{m+2}) = \delta^+$, for otherwise we get $d^{++}(v) \geq \delta^+$, which is a contradiction.\\
If $N^+_R(r_{m+2}) = \{r_1\}$, then $ r_2 \in N^{++}(r_{m+2})$.
Let $X = N^+(r_{m+2})$.
Assume that $d^+_X(z_1) \leq 1 $. Note that, $z_1 \nrightarrow r_2$. Hence, $d^{++}(r_{m+2}) \geq d^+_{\overline{X}}(z_1) + |\{r_2\}| \geq \delta^+ = d^+(r_{m+2})$, and $r_{m+2}$ is a Seymour vertex.
Assume now that $d^+_X(z_1) \geq 2$. Hence $z_1 \rightarrow z_2$ for some $z_2 \in N^+_{\overline{R}}(r_{m+2}) $.
Observe that $z_2 \nrightarrow z_3$ for any $z_3 \in N^+_{\overline{R}}(r_{m+2})$, and $ z_2 \nrightarrow r_1$ as otherwise $r_{m+2} r_1 \cdots r_{m+1} z_1 z_2 z_3$, and  $v r_{m+2} z_1 z_2 r_1 \cdots r_{m+1}$ would extend to a $C(k,1)$, which is a contradiction.
Thus $d^+_X(z_2) = 0$. Therefore $r_{m+2}$ is a Seymour vertex by Lemma~\ref{l.1}.\\
If $N^+_R(r_{m+2}) = \{r_2\}$, then $ r_3 \in N^{++}(r_{m+2})$.
Let $X = N^+(r_{m+2})$.
Suppose first that $d^+_X(z_1) \leq 1 $. We have $z_1 \nrightarrow r_3$, for otherwise $ r_3 \cdots r_{m+1} z_1r_3$ would be a directed cycle of length at most $m$, a contradiction. Thus, $d^{++}(r_{m+2}) \geq d^+_{\overline{X}}(z_1) + |\{r_3\}| \geq \delta^+ = d^+(r_{m+2})$. Therefore $r_{m+2}$ is a Seymour vertex.
Suppose now that $d^+_X(z_1) \geq 2$. Hence $z_1 \rightarrow z_2$ for some $z_2 \in N^+_{\overline{R}}(r_{m+2})$. 
Note that $z_2 \nrightarrow r_3$ since otherwise $v r_1 r_2 z_1 z_2 r_3 \cdots r_{m+2}$ would extend to a $C(k,1)$, which is a contradiction.
If $d^+_X(z_2) \leq 1 $, then $d^{++}(r_{m+2}) \geq d^+_{\overline{X}}(z_2) + |\{r_3\}| \geq \delta^+ = d^+(r_{m+2})$. Hence $r_{m+2}$ is a Seymour vertex.
Suppose then that $d^+_X(z_2) \geq 2$. This implies $z_2 \rightarrow z_3$ for some $z_3 \in N^+_{\overline{R}}(r_{m+2})$. Observe that $z_3 \nrightarrow z_4$ for any $z_4 \in N^+_{\overline{R}}(r_{m+2})$ as otherwise $r_{m+2} r_2 \cdots  r_{m+1} z_1 z_2 z_3 z_4$ would be a directed path of length \(k=m+4\), giving \(z_4 \notin X\), which is a contradiction. Additionally, $ z_3 \nrightarrow r_2 $, for otherwise $ v r_3 \cdots r_{m+2} z_1 z_2 z_3 r_2$ would extend to a $C(k,1)$, which is a contradiction.
Thus $d^+_X(z_3) = 0$. Therefore $r_{m+2}$ is a Seymour vertex by Lemma~\ref{l.1}.\\

2) Finally, suppose that $d^+_R(r_{m+2})  = 2$.\\
Clearly, we have $N^+_R(r_{m+2})  = \{r_1, r_2\}$. Recall that, by Claim~\ref{cl.3}, we may assume that $N^+_R(r_2) \subseteq \{r_3,r_4,r_5\}$. Observe that $r_2 \nrightarrow r_i$ for $i \in \{4,5\}$, since otherwise $r_2 r_i r_{i+1} \cdots r_{m+2} r_2$ would be a directed cycle of length at most $m$, which is a contradiction. Thus $N^+_R(r_2) = \{r_3\}$. It suffices to assume that $d^+(r_2)=\delta^+$. 
Let $X = N^+(r_2)$. Consider $z_1 \in N^+_{\overline{R}}(r_2,r_{m+1}, r_{m+2})$. Notice that $z_1 \nrightarrow r_3$, for otherwise $z_1 r_3 \cdots r_{m+1} z_1$ would be a directed cycle of length at most $m$, a contradiction.
If $d^+_X(z_1) = 0$, then $r_2$ is a Seymour vertex by Lemma~\ref{l.1}. Hence we may assume that $d^+_X(z_1) \geq 1$. So $ z_1 \rightarrow z_2 $ for some $z_2 \in N^+_{\overline{R}}(r_2)$.
Moreover, we have $z_2 \nrightarrow r_3$, for otherwise $v r_1 r_2 z_1 z_2 r_3 \cdots r_{m+1} r_{m+2}$ would extend to a $C(k,1)$, which is a contradiction.
If $d^+_X(z_2) = 0$, then $r_2$ is a Seymour vertex by Lemma~\ref{l.1}. Hence we may assume that $d^+_X(z_2) \geq 1$. Thus $z_2 \rightarrow z_3 $ for some $z_3 \in N^+_{\overline{R}}(r_2)$.
Now observe that, $z_3 \nrightarrow z_4$ for all $z_4 \in N^+_{\overline{R}}(r_2)$ since otherwise $r_2 r_3 \cdots r_{m+2} z_1 z_2 z_3 z_4$ would extend to a $C(k,1)$, which is a contradiction. Notice that, $z_3 \nrightarrow r_3$ since otherwise $ v r_{m+2} r_2 z_1 z_2 z_3 r_3 \cdots r_{m+1}$ would extend to a $C(k,1)$, which is a contradiction.
Hence $d^+_X(z_3) =0$, and therefore $r_2$ is a Seymour vertex by Lemma~\ref{l.1}.
\end{proof}

By Proposition~\ref{p.2} and Theorem~\ref{t.main}, we have the following.

\begin{corollary}
Let $D$ be a 7-anti-transitive oriented graph on $n$ vertices. If the minimum out-degree
and the minimum in-degree of $D$ are both at least $\frac{n}{3}$, then $D$ has a directed triangle.
\end{corollary}


\end{document}